\documentclass{amsart}
\usepackage{color}

\newtheorem{theorem}{Theorem}[section]
\newtheorem{lemma}[theorem]{Lemma}

\newcommand{\height}{\hbox{\rm H}}

\theoremstyle{definition}

\newtheorem{example}[theorem]{Example}

\theoremstyle{remark}
\newtheorem{remark}[theorem]{Remark}

\numberwithin{equation}{section}

\newcommand\RR{\mathbb R}

\newcommand\cO{\mathcal{O}}

\newcommand\vol{\operatorname{vol}}

\begin{document}
\title[Lifting, restricting and sifting integral points]{Lifting, restricting and sifting integral points on affine homogeneous varieties}
%\title[Three applications of uniformity]{Three applications of uniformity in counting lattice points}

\author{Alexander Gorodnik}
\address{School of Mathematics \\ University of Bristol \\ Bristol BS8 1TW, U.K.}
\email{a.gorodnik@bristol.ac.uk}
\thanks{The first author was supported by RCUK, EPSRC, ERC.
The second author was supported by ISF grant}

\author{Amos Nevo}
\address{Department of Mathematics, Technion }

%\curraddr{Department of Mathematics, Technion}
\email{anevo@tx.technion.ac.il}

 %   General info
%\subjclass{Primary 22D40; Secondary 22E30, 28D10, 43A10, 43A90}

%\date{September 1, 2010}

\dedicatory{}


\begin{abstract}

In \cite{GN2}  an effective solution of the lattice point counting problem in general domains in semisimple $S$-algebraic groups and affine symmetric varieties was established.  The method relies on the mean ergodic theorem for the action of $G$ on $G/\Gamma$, and implies uniformity in counting over families of lattice subgroups admitting a uniform spectral gap. 
In the present paper we extend some methods developed in \cite{NS} and use them to establish several useful  consequences of this property, including 
\begin{enumerate}
\item Effective upper bounds on lifting for solutions of congruences in affine homogeneous varieties, 
\item Effective upper bounds on the number of integral points on general subvarieties of semisimple group varieties, 
\item Effective lower bounds on the number of almost prime points on
  symmetric varieties,
\item Effective upper bounds on almost prime solutions of Linnik-type congruence problems in  homogeneous varieties.
\end{enumerate}

%An effective  solution of the lattice point counting problem in general domains in semisimple $S$-algebraic groups and affine symmetric varieties was recently established, with  the additional feature of uniformity in counting over families of lattice subgroups admitting a uniform spectral gap. 
%In the present paper we will establish several useful  consequences of this property, including 
%\begin{enumerate}
%\item Effective lifting for solutions of congruences in affine homogeneous varieties, 
%\item Effective upper bounds on the number of lattice points on general subvarieties of a semisimple group variety, 
%\item Effective lower bounds on the number of almost prime points in a family of increasing domains on symmetric varieties.
%\end{enumerate}

%

%In \cite{GN2} a solution to the
%lattice point counting problem in semisimple algebraic group and symmetric varieties
%which provides uniform estimates over families of congruence subgroups.
%In the present paper we establish several new and useful consequences of the uniformity in
%counting for congruence subgroups:
%\begin{itemize}
%\item effective lifting of integral solutions of congruences,
%\item uniform upper estimates on the number of integral points on subvarieties,
%\item sharp lower estimates on the number of almost prime points on varieties.
%\end{itemize}

\end{abstract}

\maketitle

%{\small \tableofcontents}

\section{Introduction and Statement of results}

%Let $G$ be a semisimple $S$-algebraic group (defined over $\QQ$, say), and $\Gamma$ an $S$-arithmetic lattice. 
%Under natural conditions, the unitary representations of $G$ on the spaces $L^2_0(G/\Gamma(m))$, as $\Gamma(m)$ varies over the congruence subgroups of $\Gamma$, satisfy a uniform 
%spectral bound, namely admit a uniform spectral gap. In \cite{GN2} it was shown how to use this fact to deduce that the lattice point counting problem for the lattices $\Gamma(m)\subset G $ 
%admits quantitative solution with uniform error estimate. Our purpose in the present paper is to present some useful consequences of this result, which include the following.   

Throughout the paper, $F$ denotes a number field, and  $V_F$ denotes the set of absolute values
of $F$ extending the standard normalised absolute values of the rational numbers. $F_v$, $v\in V_F$,
will denote the corresponding local fields.

We introduce local and global heights.  
For Archimedean $v\in V_F$, and for $x=(x_1,\dots,x_d)\in F_v^d$, we set 
$$
\height_v(x)=\left(|x_1|^2_v+\cdots + |x_d|^2_v \right)^{1/2},
$$
and for non-Archimedean $v$, 
$$
\height_v(x)=\max\{|x_1|_v,\ldots,|x_d|_v\}.
$$
For $x=(x_1,\dots,x_d)\in F^d$, we set
$$
\height(x)=\prod_{v\in V_F} \height_v(x).
$$

\subsection{Effective lifting of solutions of congruences}\label{sec:lift}

Let $S$ be a finite subset of $V_F$ containing all Archimedean absolute values, and 
$${O}_S=\{x\in F:\, |x|_v\le 1\hbox{ for $v\notin S$}\}$$ is the ring of $S$-integers in $F$. 
We consider a system ${\sf X}$ of polynomial equations with coefficients in ${O}_S$.
Given an ideal $\mathfrak{a}$ of ${O}_S$, we denote by ${\sf X}^{(\mathfrak{a})}$
the system of polynomial equations over the factor-ring $O_S/\mathfrak{a}$
obtained by reducing ${\sf X}$ modulo $\mathfrak{a}$.
There is a natural reduction map 
$$
\pi_{\mathfrak{a}}:{\sf X}(O_S)\to {\sf X}^{(\mathfrak{a})}(O_S/\mathfrak{a}).
$$
The question whether a solution in ${\sf X}^{(\mathfrak{a})}(O_S/\mathfrak{a})$
can be lifted to an integral solution in ${\sf X}(O_S)$ is of fundamental importance
in number theory. It is closely related to the strong approximation property for algebraic
varieties (see \cite[\S7.1]{PR}). For instance, if ${\sf G}$ is a connected $F$-simple simply connected
algebraic group which is isotropic over $S$, then ${\sf G}$ satisfies the strong approximation property
(see \cite[\S7.4]{PR}) and, in particular, the map $\pi_{\mathfrak{a}}$ is surjective in this case.
For more general homogeneous varieties, 
the map $\pi_{\mathfrak{a}}$ need not be surjective, but the image 
$\pi_{\mathfrak{a}}({\sf X}(O_S))$ can be described
using the Brauer--Manin obstructions (see \cite{har,ctx,BD}).

In this paper, we consider the problem whether  
a solution in ${\sf X}^{(\mathfrak{a})}(O_S/\mathfrak{a})$
can be lifted to an integral solution in ${\sf X}(O_S)$ {\it effectively} :
given $\bar x\in {\sf X}^{(\mathfrak{a})}(O_S/\mathfrak{a})$,
can one find $x\in {\sf X}(O_S)$, with $\height(x)$ bounded in terms of $|O_S/\mathfrak{a}|$,
such that $\pi_{\mathfrak{a}}(x)=\bar x$?

We give a positive answer to this question for
affine homogeneous varieties of semisimple groups.
Let ${\sf X}\subset F^n$ be an affine variety defined over $F$ and equipped with
a transitive action (defined over $F$) of a connected simply connected $F$-simple algebraic group ${\sf G}\subset \hbox{GL}_m$. 

Let $S$ be a finite subset of $V_F$, containing
all Archemedean absolute values, such that action of
$\sf G$ on $\sf X$ is defined over $O_S$, and $\hbox{Lie}({\sf G})\cap \hbox{M}_m(O_S)$
has a basis over $O_S$ as an $O_S$-module.
We note that every sufficiently large subset $S$ of $V_F$ satisfies the above assumptions.
Moreover, the second assumption on $S$ is satisfied when $O_S$ is a principal ideal domain.
In particular, the second assumption always holds when the field $F$ has class number one.

\begin{theorem}\label{th:lift varieties}
There exist $q_0,\sigma>0$ 
such that for every ideal $\mathfrak{a}$ of ${O}_S$
satisfying $|{O}_S/\mathfrak{a}|\ge q_0$
and every $\bar x \in \pi_{\mathfrak{a}}({\sf X}({O}_{S}))$, there exists 
$x\in {\sf X}({O}_{S})$ such that 
\begin{equation}\label{eq:cong}
 \pi_{\mathfrak{a}}(x)=\bar x \quad\hbox{and}\quad \height(x)\le  |{O}_S/\mathfrak{a}|^{\sigma}.
\end{equation}
\end{theorem}

The parameter $\sigma$ in (\ref{eq:cong}) can be explicitly computed.
For instance, for group varieties, an explicit value of $\sigma$
is given in Theorem \ref{th:lifting group} below.
The parameter $q_0$ is computable too (see Remark \ref{r:q_0} below).

\begin{remark}
The finiteness of the exponent $\sigma$  for the case of $S$-integral points in the group variety  follows from the fact that the Cayley graphs ${\sf G}^{(\mathfrak{a})}(O_S/\mathfrak{a})$
have logarithmic diameter.  The bound provided by this approach depends on a choice of generating set of ${\sf G}(O_S)$, 
and when measured in terms of the height $\height$ is of lesser quality than 
the estimate on $\sigma$ which is developed below explicitly in terms of 
geometric and representation-theoretic data of ${\sf G}$. We thank Peter Sarnak for this remark. 
\end{remark}

Let us now consider the case of a connected $F$-simple simply connected
algebraic group ${\sf G}\subset \hbox{GL}_m$ which is isotropic over $S$. 
Then it is known to satisfy the strong approximation property, and our method
gives an asymptotic formula for the number of solutions of (\ref{eq:cong}):

\begin{theorem}\label{th:lift quant}
For every $\sigma>\sigma_0$ (as in (\ref{eq:sigma_0}) below), every ideal $\mathfrak{a}$ in $O_S$
and every  $\bar x \in {\sf G}^{(\mathfrak{a})}(O_S/\mathfrak{a})$,
\begin{align*}
&\left|\left\{x\in {\sf G}(O_S);\,\, \pi_{\mathfrak{a}}(x)=\bar x,\, \height(x)\le 
|O_S/\mathfrak{a}|^{\sigma}\right\}\right|\\
=&
\frac{1}
{|{\sf G}^{(\mathfrak{a})}(O_S/\mathfrak{a})|}\cdot 
\left|\left\{x\in {\sf G}(O_S);\,\, \height(x)\le 
|O_S/\mathfrak{a}|^{\sigma}\right\}\right|\left(1+
O_\epsilon\left(|O_S/\mathfrak{a}|^{\dim({\sf G})(1-\sigma_0^{-1}\sigma)+\epsilon}\right)\right)
\end{align*}
for every $\epsilon>0$.
\end{theorem}

%Theorem \ref{th:lift quant 0} shows that integral points are equidistributed cosets.
%This is analogous to the equidistribution result over finite fields proved in \cite{luo}.

This result indicates that the properties
$\pi_{\mathfrak{a}}(x)=\bar x$ and $\height(x)\le |O_S/\mathfrak{a}|^{\sigma}$
are asymptotically independent.

\vspace{0.2cm}

We illustrate our results on a classical example --- 
the problem of representation a quadratic form by another quadratic form
(see, for instance, \cite{om}).

\begin{example}\label{ex:lift}
Let $A$ be an integral nondegenerate symmetric $(n\times n)$-matrix and 
$B$ be an integral nondegenerate symmetric $(m\times m)$-matrix with $n\le m$.
The variety 
\begin{equation}\label{eq:X0}
{\sf X}=\{x\in \hbox{M}_{m\times n}(\mathbb{C}):\,  {}^t x B x=A\}
\end{equation}
parametrises all possible representation of the quadratic form corresponding to $A$
by the quadratic form corresponding to $B$. For simplicity, we assume that $m-n\ge 3$
and $A$ is isotropic over $\mathbb{R}$. Then if the equation ${}^t x B x=A$ has a solution over $\mathbb{R}$
and over $\mathbb{Z}_p$ for every $p$, then it has an integral solution, and the reduction map
${\sf X}(\mathbb{Z})\to {\sf X}(\mathbb{Z}/q)$ is surjective for every $q\ge 1$ (see \cite[Ch.~X]{om}).
Our results implies that under the same assumptions,
for every $q\ge 1$ and $\bar x\in \hbox{M}_{m\times n}(\mathbb{Z}/q)$ satisfying
$$
{}^t \bar x B \bar x=A\,\hbox{\rm mod}\, q,
$$
there exists $x \in \hbox{M}_{m\times n}(\mathbb{Z})$ such that
\begin{equation}\label{eq:X1}
{}^t x B x=A,\quad\quad x=\bar x\,\hbox{\rm mod}\, q,\quad\quad \height_\infty(x)\ll q^{\sigma},
\end{equation}
where $\sigma>0$ is a computable constant. For instance, when $B$ has the signature
$(\lfloor m/2\rfloor,m-\lfloor m/2\rfloor)$, this estimate holds for
\begin{equation}\label{eq:sigma_m}
\sigma>\sigma_m=\left\{
\begin{tabular}{ll}
$\frac{4m(m^2-m+1)n_e}{m-1}$, &\hbox{when $m$ is odd,}\\
$\frac{4(m-1)(m^2-m+1)n_e}{m+2}$, & \hbox{when $m$ is even,}
\end{tabular}
\right.
\end{equation}
where $n_e$ denotes the least even integer $\ge \lfloor m/2\rfloor$.
We will provide details of this computation in Section \ref{sec:lifting}.
\end{example}

The following example demonstrates that the polynomial bound established in Theorem
\ref{th:lift varieties} does not hold for other homogeneous varieties.

\begin{example}\label{ex:contr}
Let $F$ be a number field of degree $d$
with an infinite group of units and $\{\xi_1,\ldots,\xi_d\}$ be a basis of the ring ${O}$ of integers
of $F$. We consider the integral polynomial
$$
f(x_1,\ldots, x_d)=N_{K/\mathbb{Q}}(x_1\xi_1+\cdots+x_d\xi_d)
$$
and the variety ${\sf X}=\{f=1\}$. Note that the set of integral points on ${\sf X}$ 
is exactly the group $U$ of units in the number field $F$. We note that
\begin{equation}\label{eq:cont1}
|\{x\in {\sf X}(\mathbb{Z}): \height_\infty(x)\le T\}|\ll (\log T)^{r+s-1}
\end{equation}
where $r$ and $s$ denote the number of real and complex absolute values of $F$ respectively.
This claim can be checked by representing the group of units as a lattice
in $\mathbb{R}^{r+s-1}$, similarly to the proof of  Dirichlet's theorem.

We also note that there are infinitely many primes $\mathfrak{p}$ in the ring of integers $O$ of $F$ such
that 
\begin{equation}\label{eq:artin}
(O/\mathfrak{p})^\times=U \,\hbox{\rm mod}\, \mathfrak{p}.
\end{equation}
It was proved in \cite{cw} that the set of such primes has positive density 
if one assumes the Generalized Riemann Hyposesis, and in \cite{n} that 
in most cases (for instance, when $[F:\mathbb{Q}]>3$), there are infinitely many such primes unconditionally.
Now it follows from \eqref{eq:artin} that
\begin{equation}\label{eq:cont2}
|\pi_p({\sf X}(\mathbb{Z}))|\ge p-1
\end{equation}
for infinitely many prime numbers $p$. Comparing \eqref{eq:cont1} and \eqref{eq:cont2},
we conclude that the polynomial bound as in Corollary \ref{th:lift varieties} is impossible in this case.
\end{example}

\subsection{Integral points on subvarieties}
We now turn to consider the problem of bounding the number of integral points
on algebraic varieties. This has been an active field of research in recent years, and 
we refer the reader to the survey \cite{hb0} and the book \cite{br} for overviews
of results and conjectures concerning upper estimates on the number of
integral points. We will concentrate on homogeneous varieties and our methods and results are motivated by those developed in \cite[\S 4.3]{NS} . 
%When $\sf X$ is a proper integral variety in the projective
%space $\mathbb{P}^N$, it is expected (see \cite{hb,bhbs}) that
%\begin{equation}\label{eq:brown}
%|\{x\in {\sf X}(\mathbb{Q}):\, \height(x)\le T\}|\ll_{\mathbb{P}^N,\deg({\sf X}),\epsilon}T^{\dim({\sf X})+\epsilon}
%\end{equation}
%for every $\epsilon>0$, provided that $\deg({\sf X})\ge 2$.

Given an affine variety $\sf X$ defined over a number field $F$, we set
$$
N_T({\sf X}({O}_S))=|\{x\in {\sf X}({O}_S):\, \height(x)\le T\}|
$$
where ${O}_S$ is a ring of $S$-integers in $F$.
The problem we will focus on is establishing an upper estimate on 
$N_T({\sf Y}({O}_S))$ for arbitrary proper affine subvarieties $\sf Y$ of $\sf X$. We will prove a {\it non-concentration phenomenon} for the collection of proper subvarieties of a semisimple group variety $\sf G$, namely that the number of $S$-integral points on a subvariety $\sf{Y}$ has strictly lower rate of growth than the group variety $\sf G$. 
We remark that this important property does not hold for general irreducible varieties $\sf X$. Indeed a bound of the form 
$$
N_T({\sf Y}({O}_S))\ll_{{\sf X},\deg({\sf Y})} N_T({\sf X}({O}_S))^{1-\sigma_{\sf Y}}
$$
with $\sigma_{\sf Y}>0$, where we write $\deg({\sf Y})$ for the degree of the projective closure of ${\sf
  Y}$, is false in general. This can be demonstrated by the variety
$x_1^3+x_2^3+x_3^3+x_4^3=0$, where most of rational point lie on lines
(see \cite{hb00}). But in the case of $S$-algebraic group varieties
we have the following:

\begin{theorem}\label{th:upper0}
Let ${\sf G}$ be a connected $F$-simple simply connected algebraic group defined over a number field $F$.
Let $S\subset V_F$ be a finite subset containing all Archimedean absolute values such that $\sf G$
is isotropic over $S$. Then there exists $\sigma=\sigma({\sf G},S,\dim({\sf Y}))\in (0,1)$
such that for every absolutely irreducible proper
affine subvariety ${\sf Y}$ of ${\sf G}$ defined over $F$, we have
$$
N_T({\sf Y}({O}_S))\ll_{{\sf G},\deg({\sf Y})} N_T({\sf G}({O}_S))^{1-\sigma}\,\,.
$$
\end{theorem}

An explicit formula for the exponent $\sigma$ is given in Theorem \ref{th:upper} below, demonstrating 
that $\sigma$ depends only on $\dim {\sf Y}$, and increases monotonically 
with the codimension of $\sf Y$. 

\vspace{0.2cm}

To demonstrate Theorem \ref{th:upper} let us consider the case of 
integral points on subvarieties of the special linear group $\hbox{SL}_n$, $n \ge 2$.

\begin{example} \label{ex:sub}
For every absolutely irreducible proper affine
subvariety $\sf Y$ of $\hbox{SL}_n$ defined over $F$, we have 
\begin{equation}\label{eq:sl}
N_T({\sf Y}(\mathbb{Z}))\ll_{n,\deg({\sf Y}),\epsilon} T^{n^2-n-\frac{n^2-1-\dim ({\sf
      X})}{(n^2+n)2n_e}+\epsilon},
\quad \epsilon>0,
\end{equation}
as $T\to\infty$, where $n_e$ is the least even integer $\ge n-1$.
This improves the trivial estimate $N_T({\sf Y}(\mathbb{Z}))\ll T^{n^2-n}$. 
Details of this computation will be given in Section \ref{sec:sub}.
\end{example}

We note that the assumption of absolute irreducibility is not crucial for the conclusion of Theorem \ref{th:upper0}. Another version of Theorem \ref{th:upper0} which 
can be proved using the argument of \cite[Lem.~4.2]{NS} is as follows. 

\begin{theorem} \label{th:upper00}
With notation as in Theorem \ref{th:upper0}, for every proper
affine subvariety ${\sf Y}$ of ${\sf G}$ defined over $F$, we have
$$
N_T({\sf Y}({O}_S))\ll_{{\sf G},{\sf Y}} N_T({\sf G}({O}_S))^{1-\sigma}\,\,.
$$
\end{theorem}

Let now ${\sf Y}_i$, $1\le i \le k$ be a collection of $k$ hypersurfaces in $\sf G$. Since the number of lattice points  in each hypresurface has lower  rate of growth than the number of lattice points in $\sf G$, the same holds for their union. Thus, the rate of growth of the number of lattice points in the complement 
of these hypersurfaces is the same as the rate of growth of all lattice points. This observation gives rise to host of results asserting that the set of integral points which are generic (i.e. avoid the union of the hypersurfaces) has maximal possible rate of growth.  Let us  illustrate this principle concretely by the following example.

\begin{example}
Denote by $N_T$ the number of  unimodular integral $(n\times n)$-matrices ($n\ge 3$)
of norm bounded by $T$, and by $N_T^\prime$ the number of such matrices satisfying 
\begin{itemize}
\item all the matrix entries are non-zero, 
\item all the principal minors do not vanish, 
\item all the eigenvalues are distinct, 
\item all the singular values (eigenvalues of $A^t A$) are distinct.
\end{itemize}
Then 
$$
N_T^\prime=N_T\cdot \left( 1+ 
O _\epsilon\left(T^{-\frac{1}{2n_e n(n+1)}+\epsilon}\right)\right),\quad \epsilon>0,
$$
where $n_e$ is the least even integer $\ge n-1$.
\end{example}

\subsection{Almost prime points on varieties and orbits}\label{sec:affine0}
We now turn to the question of how often a polynomial map $f:\mathbb{Z}^n\to\mathbb{Z}$
admits prime (or, more realistically, almost prime) values. This problem has long been studied 
using sieve methods (see, for instance, \cite{hr}). Recently substantial progress has been achieved 
in the papers \cite{bgs,ls,NS},  establishing results on the abundance 
of almost prime values for polynomials defined on homogeneous varieties and orbits
of linear groups. The goal of this section is to generalise one of the main results
of \cite{NS} to the setting of symmetric varieties.
 
Let $\sf G$ be a connected  $\mathbb{Q}$-simple simply connected algebraic group isotropic over $\mathbb{Q}$
and ${\sf G}\to \hbox{GL}_n$ a representation of $\sf G$ which is also defined over $\mathbb{Q}$.
Fix $v\in \mathbb{Z}^n$.
We assume that ${\sf X}={\sf G}v$ is Zariski closed, and ${\sf L}=\hbox{Stab}_{\sf G}(v)$ is connected
and has no nontrivial characters. Then the coordinate ring $\mathbb{C}[{\sf X}]$ is a unique factorisation
domain (see Lemma \ref{l:unique} below). 
Let $f$ be a regular function on ${\sf X}$ defined over $\mathbb{Q}$ such that it has a decomposition
into irreducible factors $f=f_1\cdots f_t$ where all $f_i$'s are distinct and defined over
$\mathbb{Q}$.
Let $\cO=\Gamma v$ be the orbit of $\Gamma={\sf G}(\mathbb{Z})$.
We assume that $f$ takes integral values on $\cO$ and is
weakly primitive (that is, $\gcd(f(x):x\in\cO)=1$).
The saturation number $r_0 ( \cO , f )$ of the pair
$(\cO, f )$ is the least $r$ such that the set of $x \in \cO$ for which
$f ( x )$ has at most $r$ prime factors is Zariski dense in $\sf X$, which is
the Zariski closure of $\cO$ by the Borel density theorem.
It is natural to ask whether the saturation number $r_0 ( \cO , f )$ is finite
and establish quantitative estimates on the set 
$\{ x \in \cO: f ( x ) \ \mbox{has at most} \ r \ \mbox{prime factors} \}$.

We fix a norm on $\mathbb{R}^n$ and set $\cO(T)=\{w\in\cO:\, \|w\|\le T\}$.
It was shown in \cite{NS} that when ${\sf X}\simeq{\sf G}$ is a group variety, 
the saturation number is finite and there exists explicit $r\ge 1$
such that 
\begin{equation}\label{eq:primes NS}
| \{ x \in \cO (T ): f ( x ) \ 
\mbox{has at most} \ r \ \mbox{prime factors} \}
| \, \gg \, \frac{|\cO( T )|}{(\log T )^{t(f)}}
\end{equation}
as $T\to\infty$. As remarked in \cite{bgs,NS}, the assumption that ${\sf X}\simeq{\sf G}$ is not crucial
if only finiteness of the saturation number is concerned, and $r_0 ( \cO , f )$ is finite
for general orbits.
However, the effective lower estimate \eqref{eq:primes NS} is much more demanding, and
so far it 
has only been established for 2-dimensional quadratic surfaces \cite{ls} and
for group varieties \cite{NS}. Our goal here is to prove \eqref{eq:primes NS} for general
symmetric varieties.

\begin{theorem}\label{th:primes0}
Let $\cO$ and $f$ be as above and assume in addition 
that ${\sf L}=\hbox{\rm Stab}_{\sf G}(v)$ is symmetric (that is,
$\sf L$ is the set of fixed points of an involution of $\sf G$).
Then there exists $r\ge 1$ such that 
\[
| \{ x \in \cO (T ): f ( x ) \ 
\mbox{has at most} \ r \ \mbox{prime factors} \}
| \, \gg \, \frac{|\cO ( T )|}{(\log T )^{t(f)}}
\]
as $T\to\infty$.
\end{theorem}

An explicit value of the number $r$ is given in Theorem \ref{th:primes} below.

\vspace{0.2cm}

We illustrate Theorem \ref{th:primes0} by three examples.

\begin{example}\label{ex:prime1}
Let $Q$ be a nondegenerate integral quadratic form in $n$ variables,
which is indefinite over $\mathbb{R}$. Let $v\in \mathbb{Z}^n$, $\Gamma=\hbox{Spin}(Q)(\mathbb{Z})$,
and $\mathcal{O}=\Gamma v$. If we assume that $Q(v)\ne 0$, then the stabiliser of $v$ in
$\hbox{Spin}(Q)$ is a symmetric subgroup of $\hbox{Spin}(Q)$. Moreover, we assume that
$n\ge 4$, which implies that
this stabilizer is connected and has no nontrivial characters.
Then Theorem \ref{th:primes0} applies and \eqref{eq:primes NS} holds.
An explicit estimate for the number $r$ of prime factors is as follows. If $Q$ has signature $(1,n-1)$ over
$\mathbb{R}$, then \eqref{eq:primes NS} holds with $r$ the least integer satisfying 
$$
r> \frac{9(n^2-n+2)(3n^2-3n+2)}{2n-4}\cdot n_e \cdot t(f)\deg(f),
$$
where $n_e$ is the least even integer $\ge 9(n-1)/7$.
On the other hand, if $Q$ has signature $(\lfloor n/2\rfloor ,n-\lfloor n/2\rfloor)$ over
$\mathbb{R}$, then \eqref{eq:primes NS} holds with
$r$ as above where $n_e$ is the least even integer $\ge \lfloor n/2\rfloor$.
We will explain this computations in Section \ref{sec:affine}.
\end{example}

\begin{example}\label{ex:prime2}
Let $A$ be a nondegenerate integral symmetric matrix of dimension $n$.
We say that another matrix $B$ is integrally equivalent of $A$ if there exists
$\gamma\in \hbox{SL}_n(\mathbb{Z})$ such that $B={}^t\gamma A \gamma$, and write
$B\sim_\mathbb{Z} A$. Let
$$
\mathcal{O}=\{B\in\hbox{M}_n(\mathbb{Z}):\, B\sim_\mathbb{Z} A\}.
$$
If $n\ge 3$, then Theorem \ref{th:primes0} implies estimate \eqref{eq:primes NS} with
$$
r>\frac{36 n(3n^2-2)}{n-1}\cdot n_e \cdot t(f)\deg(f),
$$
where $n_e$ is the least even integer $\ge n-1$.
This will be explained in detail in Section \ref{sec:affine}.
\end{example}

\subsection{Linnik-type congruence problems on varieties and orbits}\label{sec:linnik}
Our next aim is to discuss an analogue of Linnik theorem \cite{l1,l2} on the least prime  in an 
arithmetic progression, which states that 
there exists $c,\sigma>0$ such that for every coprime $b,q\in\mathbb{N}$
one can find a prime number $p$ such that
$$
p=b\,\hbox{\rm mod}\, q\quad\hbox{and}\quad p\le c\, q^\sigma.
$$
It is a very challenging goal to establish such a result in the setting of the previous section,
so to keep things more realistic, we settle for the  existence of solutions of polynomially bounded size  which are almost primes.

Let ${\sf G}\subset \hbox{GL}_n$ be a connected  $\mathbb{Q}$-simple simply connected algebraic group defined over $\mathbb{Q}$.
We fix $v\in \mathbb{Z}^n$ and
consider the orbit $\cO=\Gamma v$ of $\Gamma={\sf G}(\mathbb{Z})$.
Let $f:\mathcal{O}\to \mathbb{Z}$ be a polynomial map.
We assume that $f$ is weakly primitive, and
the regular function $\tilde f: {\sf G}\to \mathbb{C}$
defined by $\tilde f(g)=f(gv)$ decomposes as a product of $t$
irreducible factors which are distinct and defined over $\mathbb{Q}$.

\begin{theorem}\label{th:linnik}
There exist $q_0,r,\sigma>0$ 
(as in Theorem \ref{th:lin group} below) such that for every coprime $b,q\in\mathbb{N}$
satisfying $q\ge q_0$ and $b\in f(\mathcal{O})\,\hbox{\rm mod}\, q$, 
one can find $x\in\mathcal{O}$ satisfying 
\begin{enumerate}
\item[(i)] $f(x)$ is a product of at most $r$ prime factors,
\item[(ii)] $f(x)=b\,\hbox{\rm mod}\, q$ and $\|x\|\le q^\sigma$.
\end{enumerate}
\end{theorem}
The explicit values of $r$ and $\sigma$ are given in Theorem \ref{th:lin group} below,
and $q_0$ could be computed, in principle, as well.

\vspace{0.2cm}

Coming back to Example \ref{ex:prime1}, we conclude that 
for a polynomial function $f$ on $\hbox{M}_{m\times n}(\mathbb{C})$
satisfying above conditions, the system of equations
$$
{}^t x B x=A,\quad\quad f(x)=b\,\hbox{\rm mod}\, q,\quad\quad \height_\infty(x)\ll q^{\sigma},
$$
has a solution $x\in\hbox{M}_{m\times n}(\mathbb{Z})$
such that $f(x)$ is a product of at most $r$ prime factors,
provided that 
$$
{}^t x B x=A,\quad\quad f(x)=b
$$
has a solution modulo $q$, and $q$ is sufficiently large.
For instance, when $B$ has the signature $(\lfloor m/2\rfloor,m-\lfloor m/2\rfloor)$,
this holds for $\sigma>\sigma_m$ as in (\ref{eq:sigma_m}), and
$$
r>\frac{9\alpha_m\, \sigma}{\alpha_m\, \sigma- m(m-1)/2}\cdot \sigma_m \cdot t(f)\deg(f),
$$
where $\alpha_m=(m-1)^2/4$ for odd $m$ and $\alpha_m=m(m+2)/4$ for even $m$.

\vspace{0.2cm}

We remark that in Theorem \ref{th:linnik} we do not assume that the stabiliser of $v$ in $\sf G$ is symmetric.
Under this assumption, our method implies a result on the number of solutions:
  
\begin{theorem}\label{th:linnik2}
Under the additional assumption that $\hbox{\rm Stab}_{\sf G}(v)$ is symmetric, 
for every $\sigma>\sigma_0$
(as in (\ref{eq:sigma2})), $r$ (as in (\ref{eq:rrr})),
 and coprime $b,q\in\mathbb{N}$ satisfying $b\in f(\mathcal{O})\,\hbox{\rm mod}\, q$,
\begin{align*}
\left|\left\{x\in\mathcal{O}(q^\sigma):\,   
\begin{tabular}{l}
\hbox{$f(x)$ has at most $r$ prime factors}\\
\hbox{$f(x)=b\,\hbox{\rm mod}\, q$}
\end{tabular}
\right\}\right|
\gg_{\sigma} \frac{1}
{|f(\mathcal{O})\,\hbox{\rm mod}\, q|}\cdot 
\frac{|\mathcal{O}(q^\sigma)|}{(\log q)^{t(f)}}
\end{align*}
for sufficiently large $q$.
\end{theorem}

\subsection{The fundamental lattice point counting result}\label{sec:fundamental}

We now state the uniform solution given in \cite[Th.~5.1]{GN2}  to the lattice point counting problem, which underlies the results in the present paper.  
Let ${\sf G}$ be a connected $F$-simple simply connected algebraic group
defined over a number field $F$, $S$ a finite subset of $V_F$, and $O_S$ the ring of $S$-integers.
Let $G =\prod_{v\in S} {\sf G}(F_v)$, $\Gamma ={\sf G}(O_S)$, and 
$$\Gamma(\mathfrak{a}) =\{\gamma\in \Gamma:\, \gamma=I\,\hbox{\rm mod}\, \mathfrak{a}\}\quad
\hbox{for an ideal $\mathfrak{a}$ of $O_S$}\,.$$
We shall use the following notation throughout the paper :
\begin{align*}
p_S&=\hbox{the least number such that all $L^2_0(G/\Gamma(\mathfrak{a}))$ are $L^{p_S^+}$-integrable (\cite[Def.~5.2]{GN1}),}\\
n_e(p) &=
\hbox{the least even integer $\ge p/2$, if $p>2$, and $1$, if $p=2$,}\\
d_S&=\sum_{v\in V_\infty} \dim {\sf G}(F_v),\\
B_T&=\{g\in G:\, \height(g)\le T\},\\
a_S&=\hbox{the H\"older exponent of the family of sets $B_{e^t}$
(see \cite[Def. 3.12]{GN1}).}
\end{align*}
We note that the finiteness of $p_S$ is a manifestation of 
property $(\tau)$, established in full generality by Clozel \cite{cl}.  H\"older-admissibility of the 
sets $B_{e^t}$ was established in 
\cite[Th.~7.19]{GN1} and \cite{bo}. In many cases, one can take $a_S=1$ (see \cite[Ch.~7]{GN1}).
For instance, this is the case when $F=\mathbb{Q}$ and $S=\{\infty\}$.

We can now state :

%The main ingredient of the proof of Theorem \ref{th:lift varieties} (as well as
%Theorem \ref{th:upper0}) is the uniform asymptotic estimate on the number lattice points in congruence
%subgroups (see \cite[Th.~5.1]{GN2}), which we now state.

\begin{theorem}[\cite{GN2}]\label{th:S-arith}
For every $\gamma_0\in\Gamma$ and
all ideals $\mathfrak{a}$ of $O_S$, 
\begin{align*}
|\gamma_0\Gamma(\mathfrak{a})\cap B_T|&=\frac{\vol(B_T)}{[\Gamma:\Gamma(\mathfrak{a})]}+O_\epsilon\left(\vol(B_T)^{1-(2n_e(p_S))^{-1}
      a_S/(a_S+d_S)+\epsilon}\right),\quad \epsilon>0,
\end{align*}
where the Haar measure on $G$ is normalised so that $\vol(G/\Gamma)=1$.
\end{theorem} 

We set
\begin{align}\label{eq:alpha}
\alpha_S(G) &=\limsup_{T\to\infty} \frac{\log |\{\gamma\in \Gamma: \height(\gamma)\le T\}|}{\log
  T}=\limsup_{T\to\infty} \frac{\log \vol(B_T)}{\log
  T}.
\end{align}
We note that $\alpha_S(G) >0$ provided that $\sf G$ is isotropic over $S$ (see \cite[Sec.~7]{gw},
\cite{mau}, \cite[Sec.~6]{gos}). %When the context is clear, we will also denote $\alpha_S(G)=\alpha_S$. 

We will also have occasion below to consider the volume growth in a homogeneous space $G/H$, in which case we will denote the exponent by $\alpha({G/H})$. 

Although the asymptotics of $|\{\gamma\in \Gamma: \height(\gamma)\le T\}|$ is also
known, this will not be needed in our argument.

\subsection*{Acknowledgement.}
We would like to express our gratitude to Peter Sarnak
who was involved in this project at its initial stage and greatly contributed to it. 
It is a pleasure to thank Peter for generously sharing his ideas with us.

\section{Effective lifting of solutions of congruences}\label{sec:lifting}

We first establish a version of Theorem \ref{th:lift varieties} in the case of group varieties,
and  Theorem \ref{th:lift varieties} will be deduced from the following result.

\begin{theorem}\label{th:lifting group}
Let ${\sf G}\subset \hbox{\rm GL}_m$ be a connected simply connected $F$-simple algebraic group, 
and let $S$ be a finite subset of $V_F$, containing
all Archimedean absolute values, such that  
$\sf G$ is isotropic over $S$ and $\hbox{\rm Lie}({\sf G})\cap \hbox{\rm M}_m(O_S)$
has a basis over $O_S$ as an $O_S$-module.
Let 
\begin{equation}\label{eq:sigma_0}
\sigma>\sigma_0:=\alpha_S(G)^{-1} \dim({\sf G}) \frac{a_S+d_S}{a_S} 2n_e(p_S).
\end{equation}
Then  there exists $q_0>0$ such that
for every ideal $\mathfrak{a}$ of ${O}_S$ satisfying $|O_S/\mathfrak{a}|\ge q_0$ and
every $\bar x \in {\sf G}^{(\mathfrak{a})}(O_S/\mathfrak{a})$, there exists 
$x\in {\sf G}(O_S)$ such that 
\begin{equation}\label{eq:eeee}
 \pi_{\mathfrak{a}}(x)=\bar x \quad\hbox{and}\quad \height(x)\le |O_S/\mathfrak{a}|^\sigma.
\end{equation}
\end{theorem}

We note that the exponent $\sigma$ can be further improved, for instance, by considering
a smooth density on the sets $\{\height\le T\}$, and when $p_S=2$, this leads
to essentially optimal bound on $\sigma$. However, we do not pursue this direction
in the present papers and rely only on the counting estimate of Theorem \ref{th:S-arith}.

\begin{proof}
Since $\sf G$ is isotropic over $S$, it satisfies the strong approximation property with respect to $S$
(see \cite[\S7.4]{PR}). Then
 it follows that the map $\pi_{\mathfrak{a}}$ is surjective,
and there exists $\gamma_0\in \Gamma$ such that $\bar x=\pi_{\mathfrak{a}}(\gamma_0)$ for some $\gamma_0\in
\Gamma$. Moreover, we have $\bar x=\pi_{\mathfrak{a}}(\gamma_0\Gamma(\mathfrak{a}))$.

By Theorem \ref{th:S-arith}, for every $\delta< \delta_0=(2n_e(p_S))^{-1}a_S/(a_S+d_S)$ and $c_\delta>0$, we have
\begin{equation}\label{eq:lift}
\left| |\gamma_0\Gamma(\mathfrak{a})\cap
  B_T|-\frac{\vol(B_T)}{|\Gamma:\Gamma(\mathfrak{a})|}\right|\le c_\delta\,\vol(B_T)^{1-\delta}.
\end{equation}
It is important to emphasize here that this estimate is uniform over all $\gamma_0\in \Gamma$
and all ideals $\mathfrak{a}$ of ${O}_S$. 
It follows from (\ref{eq:lift}) that for $T$ satisfying
\begin{align}\label{eq:lower}
\vol(B_T)> (c_\delta\,|\Gamma:\Gamma(\mathfrak{a})|)^{1/\delta},
\end{align}
there exists $x\in \gamma_0\Gamma(\mathfrak{a})\cap B_T$.
Then we have $\pi_{\mathfrak{a}}(x)=\bar x$ and $\height(x)\le T$.

Now it remains to analyse for which values of $T$ inequality (\ref{eq:lower}) holds.
By Lemma \ref{l:upper0} below,
\begin{equation}\label{eq:111}
|\Gamma:\Gamma(\mathfrak{a})|\ll |O_S/\mathfrak{a}|^{\dim ({\sf G})}.
\end{equation}
%Also, an easy ``box-counting'' argument implies that for some $c>0$,
%$$
%|\Gamma\cap B_T|\ll \vol(B_{cT}).
%$$
By \eqref{eq:alpha}, for every $\alpha<\alpha_S$ and $T>T(\alpha)$, 
\begin{equation}\label{eq:222}
\vol(B_T)\ge T^{\alpha}.
\end{equation}
Therefore, we conclude that (\ref{eq:lower}) holds for $T= |O_S/\mathfrak{a}|^\sigma$
with $\sigma>{\dim ({\sf
    G})/(\alpha\delta)}$ and sufficiently large $|O_S/\mathfrak{a}|$.
Since this is the case for every $\alpha<\alpha_S(G)$
and $\delta<\delta_0$, this concludes the proof.
\end{proof}

To complete the proof of Theorem \ref{th:lifting group}, we therefore only have to establish the following  
\begin{lemma}\label{l:upper0}
Let ${\sf G}\subset {\rm GL}_n$ be a connected algebraic group,
and let $S$ be a finite subset of $V_F$,
containing all Archemedean absolute values, such that  
$\hbox{\rm Lie}({\sf G})\cap \hbox{\rm M}_m(O_S)$ has a basis over $O_S$.
Then 
$$
|\Gamma:\Gamma(\mathfrak{a})|\asymp |O_S/\mathfrak{a}|^{\dim ({\sf G})}
$$
uniformly over ideals $\mathfrak{a}$ of ${O}_S$.
\end{lemma}

\begin{proof}
Let ${O}_v$ denote the local ring with the prime ideal $\mathfrak{p}_v$
corresponding to non-Archemedean $v\in V_F$.
It follows from the formulas for local Tamagawa measures
(see, for instance, \cite[14.2-14.3]{vosk}) and the Lang--Weil estimates \cite{lw}
that for all valuations $v$ outside $S$ and all $n\ge 0$,
\begin{equation*}
|{\sf G}^{(\mathfrak{p}_v^n)}({O}_v/\mathfrak{p}_v^n)|\asymp |{O}_v/\mathfrak{p}_v|^{n\dim({\sf
    G})}=|{O}_v/\mathfrak{p}_v^n|^{\dim({\sf G})},
\end{equation*}
and it follows the Chinese reminder theorem that for every ideal $\mathfrak{a}$ of 
${O}_S$, 
$$
|{\sf G}^{(\mathfrak{a})}({O}_S/\mathfrak{a})|\asymp |{O}_S/\mathfrak{a}|^{\dim({\sf G})}.
$$
Since the kernel of the reduction map $\pi_{\mathfrak{a}}:\Gamma\to {\sf
  G}^{(\mathfrak{a})}({O}_S/\mathfrak{a})$ is equal to $\Gamma(\mathfrak{a})$,
this implies the claim of the lemma.
\end{proof}

\begin{remark}\label{r:q_0}
The constant $q_0$ in our results can be computed in principle.
It depends on the implicit constant in Theorem \ref{th:S-arith},
which is given explicitly in \cite{GN2}, and on $T(\alpha)$ in (\ref{eq:222}).
An explicit value of $T(\alpha)$ can derived from the asymptotic formula for 
$\vol(B_T)$ (see \cite[Ch.~7]{GN1}). 
\end{remark}

\begin{proof}[Proof of Theorem \ref{th:lift varieties}]
%Passing to a simply connected cover of $\sf G$, we may assume,
%without loss of generality, that $\sf G$ is simply connected.
%We fix a finite set $S_0\subset V_F$ such that 
%the action of $\sf G$ on $\sf X$ is defined over ${O}_{S_0}$.
%Then for ideals of ${O}_S$ with $S\supset S_0$, we can define
%the reduction of the action of $\sf G$ on $\sf X$ modulo $\mathfrak{a}$.

By the Borel--Harish-Chandra theorem \cite{bh},
the set ${\sf X}({O}_{S})$ is a union of finitely many
orbits of $\Gamma={\sf G}({O}_S)$. Hence, it suffices to prove the claim for every
$\bar x\in {\sf X}^{(\mathfrak{a})}(O_S/\mathfrak{a})$ that lifts to a point $x\in {\sf X}({O}_S)$
contained in a $\Gamma$-orbit $\Gamma x_0$ for some fixed $x_0\in {\sf X}({O}_S)$.
If ${\sf G}$ is anisotropic over every $v\in S$,
then $\Gamma$ is finite, and the claim is trivial. Hence, we may assume that ${\sf G}$
is isotropic for some $v\in S$. Then Theorem \ref{th:lifting group} applies.
We have 
$$
\bar x=\pi_{\mathfrak{a}}(\gamma\cdot x_0)=\pi_{\mathfrak{a}}(\gamma)\cdot \pi_{\mathfrak{a}}(x_0)
$$ for some $\gamma\in \Gamma$.
By Theorem \ref{th:lifting group}, there exists $\gamma'\in\Gamma$ such that
$$
 \pi_{\mathfrak{a}}(\gamma')=\pi_{\mathfrak{a}}(\gamma) \quad\hbox{and}\quad \height(\gamma')\le  |O_S/\mathfrak{a}|^{\sigma}
$$
where $\sigma$ is as in Theorem \ref{th:lifting group}.
Since $\bar x=\pi_{\mathfrak{a}}(\gamma')\cdot \pi_{\mathfrak{a}}(x_0)=\pi_{\mathfrak{a}}(\gamma'\cdot x_0)$, it remains
to observe that
\begin{equation}\label{eq:N}
\height(\gamma'\cdot x_0)\ll \height(\gamma')^N
\end{equation}
for some uniform $N>0$ determined by the action.
\end{proof}

\begin{proof}[Proof of Theorem \ref{th:lift quant}]
Let $\gamma_0\in\Gamma$ be such that $\pi_{\mathfrak{a}}(\gamma_0)=\bar x$.
By Theorem \ref{th:S-arith},
$$
|\gamma_0\Gamma(\mathfrak{a})\cap B_T|
=\frac{\vol(B_T)}{|\Gamma:\Gamma(\mathfrak{a})|} 
\left(1 +O_\delta\left(\frac{|\Gamma:\Gamma(\mathfrak{a})|}{\vol(B_T)^{\delta}}\right)\right)
$$
for every $\delta< \delta_0=(2n_e(p_S))^{-1}a_S/(a_S+d_S)$.
Hence, it follows from (\ref{eq:111}), and (\ref{eq:222}) that
for every $\alpha<\alpha_S(G)$ and $T>T(\alpha)$, we have
$$
|\gamma_0\Gamma(\mathfrak{a})\cap B_T|
=\frac{\vol(B_T)}{|\Gamma:\Gamma(\mathfrak{a})|} 
\left(1 +O_\delta\left(|{O}_S/\mathfrak{a}|^{\dim({\sf G})}T^{-\alpha\delta}\right)\right).
$$
Hence, if we pick $T=|{O}_S/\mathfrak{a}|^{\sigma}$ with $\sigma>\sigma_0$ as in (\ref{eq:sigma_0})
and sufficiently large $|{O}_S/\mathfrak{a}|$, then
\begin{align*}
|\gamma_0\Gamma(\mathfrak{a})\cap B_T|
&=\frac{\vol(B_T)}{|\Gamma:\Gamma(\mathfrak{a})|} 
\left(1 +O_{\alpha,\delta}\left(|{O}_S/\mathfrak{a}|^{\dim({\sf G})-\sigma\alpha\delta}\right)\right)\\
&=\frac{\vol(B_T)}{|\Gamma:\Gamma(\mathfrak{a})|} 
\left(1 +O_\epsilon\left(|{O}_S/\mathfrak{a}|^{\dim({\sf G})-
\dim({\sf G})\sigma_0^{-1}\sigma+\epsilon}\right)\right)
\end{align*}
for every $\epsilon>0$, where we have used that $\sigma_0=\dim({\sf G})/(\alpha_S\delta_0)$.
Finally, 
to complete the proof, we note that
$$
|\Gamma:\Gamma(\mathfrak{a})|=|{\sf G}^{(\mathfrak{a})}(O_S/\mathfrak{a})|
$$
and by Theorem \ref{th:S-arith},
$$
\vol(B_T)=
\left|\left\{x\in {\sf G}(O_S);\,\, \height(x)\le 
|O_S/\mathfrak{a}|^{\sigma}\right\}\right|\left(1+
O_{\alpha,\delta}\left(|O_S/\mathfrak{a}|^{-\sigma\alpha\delta}\right)\right).
$$
\end{proof}

Coming back to Example \ref{ex:lift}, we note that the variety $\sf X$ defined in \eqref{eq:X0}
is a homogeneous space of the spinor group ${\sf G}={\rm Spin}(B)$.
We have $\dim ({\sf G})=m(m-1)/2$. The H\"older exponent is $a_S=1$ by \cite[Prop.~7.3]{GN1}.
Now we assume that $\sf G$ has maximal $\mathbb{R}$-rank
(i.e., the signature of $B$ is $(\lfloor m/2\rfloor,m-\lfloor m/2\rfloor)$).
Then the integrability exponent is $p_S=m-1$ for odd $m$ and $p_S=m$ for even $m$ by \cite{l,o}.
By \cite{drs,em}, the growth rate $\alpha_S(G)$ of integral points in ${\sf G}(\mathbb{Z})$
can be estimated in terms of volume growth of the norm balls which is computable
in terms of the root data of ${\sf G}$ (see \cite{mau,gos}).
This gives $\alpha_S(G)=(m-1)^2/4$ for odd $m$ and $\alpha_S(G)=m(m+2)/4$ for even $m$.
Hence, Theorem \ref{th:lifting group} holds with
$$
\sigma>\left\{
\begin{tabular}{ll}
$\frac{2m(m^2-m+1)n_e}{m-1}$, &\hbox{when $m$ is odd,}\\
$\frac{2(m-1)(m^2-m+1)n_e}{m+2}$, & \hbox{when $m$ is even,}
\end{tabular}
\right.
$$
where $n_e$ denotes the least even integer $\ge \lfloor m/2\rfloor$.
We note that the action of ${\rm Spin}(B)$ on ${\sf X}$ can be given
by the standard Clifford algebra construction (see \cite[Ch.~II, \S7]{di}))
which implies that \eqref{eq:N} holds with $N=2$. This explains \eqref{eq:X1}.

\section{Integral points on subvarieties}\label{sec:sub}

The following result is a precise version of Theorem \ref{th:upper0} from Introduction.
In the statement we use notation introduced in Theorem \ref{th:S-arith}.

\begin{theorem}\label{th:upper}
Let ${\sf G}$ be a connected $F$-simple simply connected algebraic group defined over a number field $F$.
Let $S\subset V_F$ be a finite subset containing all Archemedean absolute values such that $\sf G$
is isotropic over $S$.
Then for every absolutely irreducible proper
affine subvariety ${\sf Y}$ of ${\sf G}$ defined over $F$, we have
$$
N_T({\sf Y}({O}_S))\ll_{{\sf G},\deg({\sf X}),\epsilon} N_T({\sf G}({O}_S))^{1-\frac{a_S(\dim({\sf G})-\dim({\sf Y}))}{\dim ({\sf
        G})(a_S+d_S)2n_e(p_S)}+\epsilon},\quad \epsilon>0,
$$
as $T\to\infty$.
\end{theorem}

Coming back to Example \ref{ex:sub}, we note
that in this case, $\dim (\hbox{SL}_n(\mathbb{R}))=n^2-1$, and so 
$N_T(\hbox{SL}_n(\mathbb{Z}))\sim c_n T^{n^2-n}$ with $c_n>0$. Furthermore $p_S=2(n-1)$ (see \cite{drs}), and $a_S=1$ (see \cite[Prop.~7.3]{GN1}). Hence, 
estimate \eqref{eq:sl} is a special case of Theorem \ref{th:upper}.

\begin{proof}[Proof of Theorem \ref{th:upper}]
For non-Archemedean $v\in V_F$, we denote by $f_v$ the corresponding residue field and by
$\mathfrak{p}_v$ the corresponding prime ideal.

We consider the reduction ${\sf Y}^{(v)}$ of the variety $\sf Y$ modulo a valuation $v$.
Then by Noether's theorem, ${\sf Y}^{(v)}$ is absolutely irreducible for almost all $v$.
Moreover, $\dim({\sf Y}^{(v)})=\dim({\sf Y})$ and $\deg({\sf Y}^{(v)})=\deg({\sf Y})$ 
for almost all $v$ (see \cite[Sec.~1]{odoni}). Therefore, by \cite[Prop.~12.1]{gl},
we have the following estimate
\begin{equation}\label{eq:coset0}
|{\sf Y}^{(v)}(f_v)|\ll_{\deg(\sf Y)} |f_v|^{\dim ({\sf Y})},
\end{equation}
valid for almost all $v$. We observe that each fiber of the reduction map ${\sf Y}({O}_S)\to {\sf
  Y}^{(v)}(f_v)$ is contained in a coset of the subgroup $\Gamma_v=\{\gamma\in \Gamma:\,
\gamma=I\,\hbox{\rm mod}\, \mathfrak{p}_v\}$ of $\Gamma={\sf G}({O}_S)$. 
Hence, it follows that ${\sf Y}({O}_S)$ is contained in a union of
at most $O_{\deg(\sf Y)}\left(|f_v|^{\dim({\sf Y})}\right)$ cosets $\gamma\Gamma_v$ with
 $\gamma\in \Gamma$. 

The crucial ingredient of the proof is Theorem \ref{th:S-arith}, which gives an 
estimate of the number points in the cosets $\gamma\Gamma_v$ uniformly over $\gamma\in \Gamma$.
More precisely, by Theorem \ref{th:S-arith}, for all all $v$,
\begin{equation}\label{eq:coset}
|\gamma\Gamma_v\cap
B_T|=\frac{\vol(B_T)}{|\Gamma:\Gamma_v|}+O_\epsilon\left(\vol(B_T)^{1-\frac{a_S}{(a_S+d_S)2n_e(p_S)}+\epsilon}\right),\quad
\epsilon>0.
\end{equation}
For almost all $v$, the reduction ${\sf G}^{(v)}$ is smooth geometrically irreducible variety of dimension $\dim {\sf G}$.
Therefore, we have the Lang--Weil estimate (see \cite{lw})
$$
|{\sf G}^{(v)}(f_v)|=|f_v|^{\dim({\sf G})}+O_{\sf G}\left(|f_v|^{\dim ({\sf G})-1/2}\right).
$$
Since $\sf G$ is simply connected $F$-simple and isotropic over $S$, 
it follows from the strong approximation
property (see \cite[Theorem~7.12]{PR}) that the reduction map $\Gamma\to {\sf G}^{(v)}(f_v)$
is surjective for all $v\notin S$. This implies the estimate
$$
|\Gamma:\Gamma_v|= |{\sf G}^{(v)}(f_v)|\gg |f_v|^{\dim ({\sf G})}
$$
for almost all $v$.

Finally, we conclude from \eqref{eq:coset0} and \eqref{eq:coset} that
for all $v$,
$$
|{\sf Y}({O}_S)\cap B_T|\ll_{{\sf G}, \deg({\sf Y}),\epsilon} |f_v|^{\dim({\sf Y})}\left(\frac{\vol(B_T)}{|f_v|^{\dim({\sf
          G})}}+\vol(B_T)^{1-\frac{a_S}{(a_S+d_S)2n_e(p_S)}+\epsilon}\right),\quad \epsilon>0.
$$
To optimise this estimate, we take  $v$ such that
$$
\vol(B_T)^{\frac{a_S}{(a_S+d_S)2n_e(p_S)}}\le  |f_v|^{\dim ({\sf G})} \le 2\,\vol(B_T)^{\frac{a_S}{(a_S+d_S)2n_e(p_S)}}.
$$
For sufficiently large $T$, such $v$ exists by the prime number theorem for the ring of integers $O$ in $F$.
This gives the estimate
$$
N_T({\sf Y}({O}_S))=|{\sf Y}({O}_S)\cap B_T|\ll_{{\sf G}, \deg({\sf Y}),\epsilon}
\vol(B_T)^{1-\frac{a_S(\dim({\sf G})-\dim({\sf Y}))}{\dim({\sf
        G})(a_S+d_S)2n_e(p_S)}+\epsilon},\quad \epsilon>0,
$$ 
as $T\to\infty$.
Since $N_T({\sf G}({O}_S))\sim \vol(B_T)$ by Theorem \ref{th:S-arith}, this
completes the proof.
\end{proof}

\section{Almost prime points on varieties and orbits}\label{sec:affine}

We now turn to the problem of establishing the existence of almost prime points on symmetric varieties. 
We shall use the notation from Section \ref{sec:affine0}. In particular,
$\sf G$ is a connected $\mathbb{Q}$-simple simply connected algebraic group defined over $\mathbb{Q}$
and $\sf L$ is a symmetric $\mathbb{Q}$-subgroup.
Let $G={\sf G}(\mathbb{R})$ and $L={\sf L}(\mathbb{R})$. Then 
$G$ is a connected semisimple Lie group with finite center
and $L$ is a closed symmetric subgroup of $G$.
We shall use the structure theory of affine symmetric spaces (see \cite[Part II]{hs}).
Fix a maximal compact subgroup $K$ of $G$ compatible with $L$
and a Cartan subgroup $A$ for the pair $(K,L)$. Then the Cartan decomposition
$$
G=KA^+L
$$
holds where $A^+$ denotes a closed positive Weyl chamber in $A$.
Let $M$ denote the centraliser of $A$ in $K\cap L$.
We fix a bounded subset $\Psi$ of $M\backslash L$ with nonempty interior which we assume to
be Lipschitz well-rounded (in the sense of \cite[Sec.~7]{GN2}).
We also denote by $\dot{A}^+$ the interior of the Weyl chamber $A^+$ and  set
$$
S_T=\{g\in K {} \dot{A}^+\Psi:\, \|gv\|\le T\}.
$$
We note that it was shown in \cite[Prop.~8.4]{GN2} that the sets $S_{e^t}$ are H\"older well-rounded
with exponent $1/3$.

Our main tool is the following result on counting of lattice
points in $S_T$ for the congruence subgroups $\Gamma(q)=\{\gamma\in\Gamma:\,\gamma=I\,\hbox{\rm mod}\, q\}$
of $\Gamma={\sf G}(\mathbb{Z})$.
We note that representations $L^2_0(G/\Gamma(q))$ are all $L^{p+}$ 
with uniform $p>0$ by \cite{cl}, so
that the following theorem is a special case of \cite[Th.~8.1]{GN2}.

\begin{theorem}[\cite{GN2}]\label{th:GN2}
For every $\gamma_0\in\Gamma$ and $q\ge 1$, 
$$
|\gamma_0\Gamma(q)\cap S_T|=\frac{\vol(S_T)}{[\Gamma:\Gamma(q)]}+O_{\epsilon}\left(\vol(S_T)^{1-(2n_e(p))^{-1}(1+3\dim G)^{-1}+\epsilon}\right),\quad \epsilon>0,
$$
where the Haar measure is normalised so that $\vol(G/\Gamma)=1$.
\end{theorem}

We note that by \cite{drs,em}
\begin{equation}\label{eq:vol1}
|\cO(T)|\sim \frac{\vol(L/(L\cap\Gamma))}{\vol(G/\Gamma)}\cdot\vol(S_Tv)\quad\hbox{as $T\to\infty$,}
\end{equation}
where $\vol$ denote $G$-invariant measures on the corresponding spaces.
It was shown in \cite[Sec.~6]{gos} that
\begin{equation}\label{eq:vol2}
\vol(S_Tv)\sim v_0 T^{\alpha(G/H)}(\log T)^\beta\quad\hbox{as $T\to\infty$,}
\end{equation}
for some $v_0>0$, $\alpha(G/H)\in \mathbb{Q}^+$, and $\beta\in\mathbb{Z}^+$.
Also, it is clear that
\begin{equation}\label{eq:vol3}
\vol(S_T)= \vol (S_Tv)\cdot \vol(\Psi).
\end{equation}

Now we prove the following theorem, which is a more explicit  version of Theorem \ref{th:primes0}
stated in \S  \ref{sec:affine0} (we refer there for the notation used below).   
\begin{theorem}\label{th:primes}
With the notation above, let $r$ be the least integer satisfying 
$$r  >   9 \alpha(G/H)^{-1}\,( 1 + \dim (G))( 1 + 3\dim (G)) 2n_e(p)\cdot t(f) \deg (f).$$ 
Then 
\[
| \{ x \in \cO (T ): f ( x ) \ 
\mbox{has at most} \ r \ \mbox{prime factors} \}
| \, \gg \, \frac{|\cO ( T )|}{(\log T )^{t(f)}}
\]
as $T\to\infty$.
\end{theorem}

In the case of Example \ref{ex:prime1}, we have $\dim(\hbox{Spin}(Q))=n(n-1)/2$ and $\alpha(G/H)=n-2$. 
When $Q$ has signature $(n,1)$, one can take $p=9(n-1)/7$ (see \cite{bs}).
For other signatures, the group $\hbox{Spin}(Q)(\mathbb{R})$ 
has $\mathbb{R}$-rank at least 2 and we can utilise
the estimates on integrability exponents obtained in \cite{l,o}.
In particular, when $Q$ has signature $(\lfloor n/2\rfloor, n-\lfloor n/2\rfloor)$
(i.e., when $\hbox{Spin}(Q)$ is split over $\mathbb{R}$), we have
$p=n-1$ for odd $n$ and $p=n$ for even $n$.

In the case of Example \ref{ex:prime2}, we have $\dim (\hbox{SL}_n)=n^2-1$,
$\alpha=(n^2-n)/2$ (see \cite[Sec.~2.3]{gos}), and $p=2(n-1)$ (see \cite{drs}).

Before we start the proof of Theorem \ref{th:primes}, we show that the decomposition $f=f_1\cdots f_t$
into irreducible factors is well-defined.

\begin{lemma}\label{l:unique}
Let $\sf G$ be a connected semisimple simply connected algebraic group and $\sf L$ a closed connected subgroup
with no nontrivial characters. Then the coordinate ring $\mathbb{C}[{\sf G}/{\sf L}]$
is a unique factorisation domain.
\end{lemma}

\begin{proof}
We refer to \cite{fi,kkv,popov} for computation of Picard groups of homogeneous spaces.
There is an exact sequence
$$
\mathcal{X}({\sf G})\to\mathcal{X}({\sf L})\to \hbox{Pic}({\sf G}/{\sf L})\to \hbox{Pic}({\sf G})
$$
where $\mathcal{X}({\sf G})$ and $\mathcal{X}({\sf L})$ denote the character groups.
Since $\sf G$ is simply connected, $\hbox{Pic}({\sf G})=1$. Hence, it follows from
the exact sequence that $\hbox{Pic}({\sf G}/{\sf L})=1$, and $\mathbb{C}[{\sf G}/{\sf L}]$
is a unique factorisation domain by \cite[Prop.~6.2]{hart}.
\end{proof}

\begin{proof}[Proof of Theorem \ref{th:primes}]
Using the dominant map ${\sf G}\to {\sf X}$,
every element $f\in\mathbb{C}[{\sf X}]$ lifts to an element $\tilde f\in \mathbb{C}[{\sf G}]$.
Since $\sf G$ is simply connected, the ring $\mathbb{C}[{\sf G}]$ is a unique
factorisation domain. We claim that the decomposition of $\tilde f$ into irreducible factors in $\mathbb{C}[{\sf G}]$ is of the form 
$\tilde f=\tilde f_1\cdots \tilde f_t$, where $f=f_1\cdots f_t$ is the decomposition in $\mathbb{C}[{\sf X}]$. 
Indeed, suppose that $\tilde f_i=g_1\cdots g_s$ for $g_1,\ldots,g_s\in \mathbb{C}[{\sf G}]$
is the decomposition into irreducibles.
We consider the right action of $\sf L$ on 
$\mathbb{C}[{\sf G}]$. Since $\tilde f_i$ is $\sf L$-invariant and $\sf L$ is connected,
it follows from uniqueness of the decomposition that each $g_i$ is also $\sf L$-invariant and
descends to a function on $\mathbb{C}[{\sf X}]$, which implies that this decomposition
must be trivial. Hence, $\tilde f_i$'s are irreducible.

Now we apply the argument of \cite{NS} to the polynomial function $\tilde f:\Gamma\to\mathbb{Z}$
and the sets $\Gamma\cap S_T$ (instead of sets $\{\gamma\in\Gamma:\, \|\gamma\|<T\}$).
It follows from Theorem \ref{th:GN2} that for every $q\ge 1$ and $\gamma_0\in \Gamma$, 
\begin{align}\label{eq:sss}
\frac{|\gamma_0\Gamma(q)\cap S_T|}{\vol(S_T)}=
\frac{1}{[\Gamma:\Gamma(q)]}+O_\epsilon\left(\vol(S_T)^{-(2n_e(p))^{-1}(1+3\dim
    G)^{-1}+\epsilon}\right),\quad \epsilon>0.
\end{align}
Therefore, by \eqref{eq:vol2}--\eqref{eq:vol3},
\begin{align*}
\frac{|\gamma_0\Gamma(q)\cap S_T|}{\vol(S_T)}=
\frac{1}{[\Gamma:\Gamma(q)]}+O_\epsilon\left(T^{-\frac{\theta}{1+3\dim(G)}+\epsilon}\right),\quad \epsilon>0,
\end{align*}
where $\theta=\frac{\alpha(G/H)}{2n_e(p)}$. This estimate is a substitute for \cite[Th.~3.2]{NS}.
Given the estimate above for a family of sets $S_T$, the argument in \cite{NS} for norm balls can be
carried out without change, and we conclude that
for sufficiently large $T$,
\begin{equation}\label{eq:sum}
\sum_{\gamma\in \Gamma\cap S_T: \gcd(\tilde f(\gamma),P_z)=1} 1 \gg \frac{|\Gamma\cap S_T|}{(\log
  |\Gamma\cap S_T|)^{t(f)}},
\end{equation}
where 
$$
P_z=\prod_{p\le z} p,\quad  z=|\Gamma\cap S_T|^\kappa,\quad
\kappa=(9 t(f)(1+\dim(G))(1+3\dim(G))2n_e(p))^{-1}.
$$
For every $\gamma\in \Gamma\cap S_T$, we have
$$
|\tilde f(\gamma)|=|f(\gamma v)|\ll T^{\deg(f)}.
$$
On the other hand, if $\gcd(\tilde f(\gamma),P_z)=1$, then
every prime factor of $\tilde f(\gamma)$ is at least $z$,
and $z\gg T^{\alpha(G/H) \kappa}$ by \eqref{eq:sss} and \eqref{eq:vol2}--\eqref{eq:vol3}.
Therefore, for every term in the sum \eqref{eq:sum}, the number of prime factors of $\tilde f(\gamma)$
is bounded above by 
$$
\frac{\deg(f)}{\alpha(G/H)\kappa}=9 \alpha(G/H)^{-1}(1+\dim(G))(1+3\dim(G))2n_e(p)t(f)\deg(f)
$$
provided $T$ is  sufficiently large. We conclude that
\begin{equation}\label{eq:final}
|\{\gamma\in \Gamma\cap S_T:\, \hbox{$f(\gamma v)$ has at most $r$ prime factors}\}|\gg
\frac{|\Gamma\cap S_T|}{(\log |\Gamma\cap S_T|)^{t(f)}}.
\end{equation}

To finish the proof, we consider the projection map
$$
\pi:\Gamma\cap S_T\to \cO(T):\gamma\mapsto \gamma v.
$$
It follows from the uniqueness properties of the Cartan decomposition 
(see \cite[p.~108]{hs}) that
if $\gamma_0,\gamma \in \Gamma\cap S_T$ satisfy $\gamma_0v=\gamma v$,
then their $KA^+$-components are equal modulo $M$,
and $\gamma_0^{-1}\gamma\in \Psi^{-1}\Psi$. Hence,
$$
\pi^{-1}(\gamma_0 v)\subset \gamma_0 \Psi^{-1}\Psi \cap \Gamma,
$$
and the cardinality of every fiber of $\pi$ is bounded by $|\Psi^{-1}\Psi \cap \Gamma|$.
It follows from \eqref{eq:final} that
$$
|\{w\in \cO(T):\, \hbox{$f(w)$ has at most $r$ prime factors}\}|\gg
\frac{|\Gamma\cap S_T|}{(\log |\Gamma\cap S_T|)^{t(f)}}
$$
as $T\to\infty$.
Since $|\Gamma\cap S_T|\asymp \vol(S_T)$, the claim of the theorem now follows
from \eqref{eq:vol1}--\eqref{eq:vol3}.
\end{proof}

We also establish a quantitative version of Theorem \ref{th:lift varieties}
for lifting solutions of congruences in $\mathcal{O}$, which will be used to prove 
Theorem \ref{th:linnik2} in Section \ref{sec:linnik}.

\begin{theorem}\label{th:lift quant sym}
For every 
\begin{equation}\label{eq:sigma2}
\sigma>\sigma_0:=\alpha(G/H)^{-1}\dim(G) (1+3\dim(G))2n_e(p),
\end{equation}
sufficiently large $q$, and $b\in \mathcal{O}\,\hbox{\rm mod}\, q$,
\begin{align*}
\left|\left\{x\in \mathcal{O}(q^\sigma);\,\,\, x=b\,\hbox{\rm mod}\, q\right\}\right|
\gg_\sigma |\mathcal{O}(q^\sigma)|\cdot
\frac{1}{|\mathcal{O}\,\hbox{\rm mod}\, q|}.
\end{align*}
\end{theorem}

\begin{proof}
Using Theorem \ref{th:GN2} and arguing exactly as in 
the proof of Theorem \ref{th:lift quant},  we get the estimate
$$
|\gamma\Gamma(q)\cap S_T|\gg_\sigma \frac{1}{|\Gamma:\Gamma(q)|}\cdot
|\Gamma\cap S_T|.
$$
for $T=q^\sigma$ with sufficiently large $q$ and every $\gamma\in\Gamma$.
This implies that
\begin{align*}
&\left|\left\{\gamma\in \Gamma\cap S_T;\,\,\, \gamma v=b\,\hbox{\rm mod}\, q\right\}\right|
=\sum_{\gamma\in \Gamma/\Gamma(q):\, \gamma v=b\,\hbox{\rm\tiny mod}\,q} |\gamma\Gamma(q)\cap S_T|\\
\gg_\sigma & \;\frac{|\hbox{Stab}_\Gamma(b\,\hbox{\rm mod}\, q):\Gamma(q)|}{|\Gamma:\Gamma(q)|}\cdot
|\Gamma\cap S_T|
=\frac{1}{|\mathcal{O}\,\hbox{\rm mod}\, q|}\cdot |\Gamma\cap S_T|.
\end{align*}
Recall from the previous proof that the cardinality of the fibers of the map
$$
\pi:\Gamma\cap S_T\to \cO(T):\gamma\mapsto \gamma v.
$$
is uniformly bounded.
Therefore,
$$
\left|\left\{x\in \mathcal{O}(T);\,\,\, x=b\,\hbox{\rm mod}\, q\right\}\right|
\gg 
\left|\left\{\gamma\in \Gamma\cap S_T;\,\,\, \gamma v=b\,\hbox{\rm mod}\, q\right\}\right|.
$$
Since $|\Gamma\cap S_T|\asymp \left|\mathcal{O}(T)\right|$
by (\ref{eq:vol1}), this completes the proof.
\end{proof}

\section{Linnik-type congruence problems  on varieties and orbits}\label{sec:linnik}

We start by proving Theorem \ref{th:linnik} for group varieties.
Let ${\sf G}\subset \hbox{GL}_n$
be a connected $\mathbb{Q}$-simple simply connected algebraic
group defined over $\mathbb{Q}$. We assume that $\sf G$ is isotropic
over $\mathbb{R}$, and denote by $\alpha=\alpha(G)>0$ the volume growth exponent
of ${\sf G}(\mathbb{R})$, defined as in (\ref{eq:alpha}).
Let $f$ be a regular function on $\sf G$ defined over $\mathbb{Q}$
that decomposes into product of $t=t(f)$ absolutely irreducible factors defined over $\mathbb{Q}$.
We assume that $f({\sf G}(\mathbb{Z}))\subset \mathbb{Z}$ and $f$
is weakly primitive. 

\begin{theorem}\label{th:lin group}
Let 
\begin{align*}
\sigma&>\sigma_0:=\alpha^{-1}\dim({\sf G})(1+\dim({\sf G}))2n_e(p),\\
r&>\frac{9\alpha \sigma}{\alpha \sigma-\dim({\sf G})}\cdot \sigma_0 \cdot t(f)\deg(f).
\end{align*}
Then there exists $q_0>0$ such that for every coprime $b,q\in\mathbb{N}$
satisfying $q\ge q_0$ and $b\in f({\sf G}(\mathbb{Z}))\,\hbox{\rm mod}\, q$, one can find
$x\in {\sf G}(\mathbb{Z})$ such that
\begin{enumerate}
\item[(i)] $f(x)$ is a product of at most $r$ prime factors,
\item[(ii)] $f(x)=b\,\hbox{\rm mod}\, q$ and $\|x\|\le q^\sigma$.
\end{enumerate}
\end{theorem}

\begin{proof}
We write $f(x)=\frac{1}{N}g(x)$ where $g(x)$ is a polynomial
with integral coefficients and $N\in \mathbb{N}$.
Since $f$ is weakly primitive, 
\begin{equation}\label{eq:gcd0}
\gcd(g(\gamma):\gamma\in\Gamma)=N.
\end{equation}
Let $N=N_1N_2$ where $N_1$ is the product of all prime factors coprime to $q$.
Then the condition $f(\gamma)=b\,\hbox{mod}\, q$ is equivalent
to $g(\gamma)=bN\,\hbox{mod}\, qN$. Moreover, because of (\ref{eq:gcd0}),
it is equivalent to $g(\gamma)=bN\,\hbox{mod}\, qN_2$.

According to our assumptions, there exists $\gamma_0\in {\sf G}(\mathbb{Z})$ such that $f(\gamma_0)=b\,\hbox{\rm mod}\, q$.
We set 
\begin{align*}
\Gamma&={\sf G}(\mathbb{Z}),\\
\Gamma_q&=\Gamma(qN_2)=\{\gamma\in \Gamma:\, \gamma=id\,\,\,\hbox{\rm mod}\, qN_2\},\\
\mathcal{O}_q(T)&=\{\gamma\in \gamma_0\Gamma_q:\, \|\gamma\|\le T\}.
\end{align*}
Note that every $\gamma\in \gamma_0\Gamma_q$ satisfies $f(\gamma)=b\,\hbox{\rm mod}\, q$.

%In fact, we claim that 
%\begin{equation}\label{eq:gcd}
%\gcd(g(\gamma):\gamma\in\gamma_0\Gamma_a)=N.
%\end{equation}
%Indeed, for primes $p$ that divide $a$
%and $\gamma\in \gamma_0\Gamma_a$, we have $g(\gamma )=b\mod p$, so that $g(\gamma)$ is not divisible
%by $p$ since $a$ and $b$ are coprime. On the other hand, for $d$ coprime to $a$, it follows from the
%strong approximation theorem that $\Gamma_a$ surjects onto $\Gamma/\Gamma(d)$. This implies the claim.

Let $\mathcal{P}_{q,z}$ be the set of prime numbers which are coprime to $q$ and bounded by $z$.
Our aim is to estimate from below the cardinality of points $\gamma\in \mathcal{O}_q(T)$ such that
$f(\gamma)$ is coprime to $\mathcal{P}_{q,z}$, which we denote by $S(T,q,z)$.
This will achieved by applying the combinatorial sieve as in \cite[Thm~7.4]{hr} and \cite[Sec.~2]{NS}.
Let 
$$
a_k=|\{\gamma\in\mathcal{O}_q(T): f(\gamma)=k\}|\quad\hbox{and}\quad
X=|\mathcal{O}_q(T)|=\sum_{k\ge 0} a_k.
$$
In order to apply the combinatorial sieve, we need to verify the following conditions:
\begin{enumerate}
\item[($A_0$)] For every square-free $d$ in  $\mathcal{P}_{q,z}$,
\begin{equation}\label{eq:a_k}
\sum_{k=0\,\hbox{\tiny mod}\, d} a_k= \frac{\rho(d)}{d}X +R_d,
\end{equation}
where $\rho(d)$ is a nonnegative multiplicative function such that for primes $p\in
\mathcal{P}_{q,z}$, we have 
\begin{equation}\label{eq:c_1}
\frac{\rho(p)}{p}\le c_1
\end{equation}
for some $c_1<1$.

\item[($A_1$)] Summing over square-free $d$ in $\mathcal{P}_{q,z}$,
$$
{\sum_{d\le X^\tau}}^{\!\!\prime}\, |R_d|\le c_2 X^{1-\zeta}
$$
for some $c_2,\tau,\zeta>0$.

\item[($A_2$)] For every $2\le w\le z$,
\begin{equation}\label{eq:log}
-l\le \sum_{p\in \mathcal{P}_{q,z}: w\le p< z} \frac{\rho(p)\log p}{p}- t\log \frac{z}{w}\le c_3
\end{equation}
   for some $c_3,l,t>0$.
\end{enumerate}
Assuming that ($A_0$), ($A_1$), and ($A_2$), 
\cite[Th.~7.4]{hr} implies that for $z=X^{\tau/s}$ with $s>9t$, the following estimate holds:
\begin{equation}\label{eq:S}
S(T,q,z)
\ge X W(z)\left(C_1-C_2 l\frac{(\log\log 3X)^{3t+2}}{\log X}\right),
\end{equation}
where 
$$
W(z)=\prod_{p\in\mathcal{P}_{q,z}:p\le z} \left(1-\frac{\rho(p)}{p}\right),
$$
and
the constants $C_1, C_2>0$ are determined by $c_1$, $c_2$, $c_3$, $\tau$, $\zeta$, $t$.

We deduce ($A_0$) and ($A_1$) from the estimates on the cardinality of lattice points
given by Theorem  \ref{th:S-arith}. 
Let $\pi_{dN_1}:\Gamma\to\Gamma(dN_1)$ denotes the factor map.
It follows from the strong approximation property that $\gamma_0\Gamma_q$ surjects onto
$\Gamma\to\Gamma(dN_1)$ under $\pi_{dN_1}$.
We set $B_T=\{h\in {\sf G}(\mathbb{R}):\, \|h\|\le T\}$.
By Theorem \ref{th:S-arith}, for every $d$ coprime to $q$ and
$\delta\in \Gamma/\Gamma_q(dN_1)$,
we have $\Gamma_q(dN_1)=\Gamma(qdN)$ and
\begin{align*}
|\delta\Gamma_q(dN_1)\cap B_T|
&=\frac{\vol(B_T)}{[\Gamma:\Gamma(qdN)]}+O_\epsilon\left(\vol(B_T)^{1-(2n_e(p))^{-1}(1+\dim({\sf G}))^{-1}+\epsilon}\right)\\
&=\frac{\vol(B_T)}{[\Gamma:\Gamma(dN_1)]\cdot
  [\Gamma:\Gamma_q]}+O_\epsilon\left(\vol(B_T)^{1-(2n_e(p))^{-1}(1+\dim({\sf  G}))^{-1}+\epsilon}\right)\\
&=\frac{|\gamma_0\Gamma_q\cap
  B_T|}{[\Gamma:\Gamma(dN_1)]}+O_\epsilon\left(\vol(B_T)^{1-(2n_e(p))^{-1}(1+\dim({\sf  G}))^{-1}+\epsilon}\right)\\
&=\frac{X}{|{\sf G}(\mathbb{Z}/(dN_1))|}+
O_\epsilon\left(X^{1-(2n_e(p))^{-1}(1+\dim({\sf G}))^{-1}+\epsilon}\right)
\end{align*}
for every $\epsilon>0$. 
We note that
for $d$ coprime to $q$, we have $f(\gamma)=0\,\hbox{\rm mod}\, d$ if and only if $g(\gamma)=0
\,\hbox{\rm mod}\, dN_1$. Restricting the sums below to $d$ coprime to $q$, we have 
\begin{align*}
\sum_{k=0\,\hbox{\tiny mod}\, d} a_k&= |\{\gamma\in \gamma_0\Gamma_q\cap B_T;\, f(\gamma)=0\,\hbox{mod}\, d\}|\\
&=\sum_{\delta\in \pi_{dN_1}(\gamma_0\Gamma_q): g(\delta)=0\,\hbox{\tiny mod}\, dN_1}
|\delta\Gamma_q(dN_1)\cap B_T| \\
&=|{\sf G}(\mathbb{Z}/(dN_1))\cap \{g=0\}|\cdot \left(\frac{X}{|{\sf
      G}(\mathbb{Z}/(dN_1))|}+O_\epsilon\left(X^{1-(2n_e(p))^{-1}(1+\dim({\sf G}))^{-1}+\epsilon}\right)\right)\\
&=\frac{\rho(d)}{d}X+
O_\epsilon\left(|{\sf G}(\mathbb{Z}/(dN_1))\cap \{g=0\}|X^{1-(2n_e(p))^{-1}(1+\dim({\sf G}))^{-1}+\epsilon}\right),
\end{align*}
where
$$
\rho(d)=\frac{d|{\sf G}(\mathbb{Z}/(dN_1))\cap \{g=0\}|}{|{\sf
      G}(\mathbb{Z}/(dN_1))|}.
$$
As in \cite[Sec.~4.1]{NS}, we deduce that $\rho$ is multiplicative function, 
(\ref{eq:c_1}) holds, and
%\begin{align*}
%|{\sf G}(\mathbb{Z}/p)|&=p^{\dim({\sf G})}+O\left(p^{\dim({\sf G})-1/2}\right),\\
%|{\sf G}(\mathbb{Z}/p)\cap \{g=0\}|&=t(f) p^{\dim({\sf G})-1}+O_f\left(p^{\dim({\sf G})-3/2}\right).
%\end{align*}
%In particular,
\begin{equation}\label{eq:rho}
\rho(p)=t(f)+O_f(p^{-1/2}).
\end{equation}
Using that
$$
|{\sf G}(\mathbb{Z}/(dN_1))\cap \{g=0\}|\ll d^{\dim ({\sf G})-1},
$$
we obtain
\begin{align*}
{\sum_{d\le X^\tau}}^{\!\!\prime}\, |R_d|\ll_\epsilon {\sum_{d\le X^\tau}} d^{\dim ({\sf G})-1} 
X^{1-(2n_e(p))^{-1}(1+\dim({\sf G}))^{-1}+\epsilon}\\
\ll (X^\tau)^{\dim({\sf G})}X^{1-(2n_e(p))^{-1}(1+\dim({\sf G}))^{-1}+\epsilon}
\ll X^{1-\zeta}
\end{align*}
for some $\zeta>0$, provided that $\tau<\tau_0=(2n_e(p))^{-1}\dim({\sf G})^{-1}(1+\dim({\sf G}))^{-1}$.
This concludes the proof of ($A_0$) and ($A_1$).

To prove ($A_2$), we observe that it follows from (\ref{eq:rho})
(see \cite[Th.~2.7(b)]{mv}) that
$$
-c_3\le \sum_{ z\le p\le w } \frac{\rho(p)\log p}{p}- t(f)\log \frac{z}{w}\le c_3
$$
for some $c_3>0$. This implies the upper estimate in (\ref{eq:log}).
The lower estimate with $l=O(\log\log q)$ follows from Lemma \ref{l:a} below.

Now it follows from (\ref{eq:S}) that
\begin{equation}\label{eq:s000}
S(T,q,X^{\tau/s})\gg \frac{X}{(\log X)^{t(f)}}\left(C_1-C_2' (\log\log q)\frac{(\log\log X)^{3t(f)+2}}{\log X}\right).
\end{equation}
Here we used that $W(z)\gg (\log z)^{-t(f)}$, which follows from (\ref{eq:rho}).

We apply (\ref{eq:s000}) with $T=q^\sigma$ with $\sigma>\sigma_0$
and sufficiently large $q$.
Then by Theorem \ref{th:lift quant}, Lemma \ref{l:upper0}, and (\ref{eq:222}),
\begin{equation}\label{eq:quant 1}
X=|\gamma_0\Gamma_q\cap B_T|\gg_\sigma \frac{\vol(B_T)}{|\Gamma:\Gamma_q|}
 \gg_{\alpha'} q^{\alpha' \sigma-\dim({\sf G})}
\end{equation}
with $\alpha'<\alpha$. Hence, for sufficiently large $q$,
\begin{equation}\label{eq:quant 2}
S(T,q,X^{\tau/s})\gg_{\sigma,\alpha'} \frac{X}{(\log X)^{t(f)}}
\end{equation}
We note that every point $\gamma$ which is counted in $S(T,q,X^{\tau/s})$ 
satisfies conclusion (ii) of the theorem, and 
$$
|f(\gamma)|\ll T^{\deg(f)}=q^{\sigma\deg(f)},
$$
and every prime $p$ which is coprime to $q$ and divides $f(\gamma)$ must
satisfy 
$$
p>X^{\tau/ s}\gg_{\sigma,\alpha'} q^{(\alpha' \sigma-\dim({\sf G}))\tau s^{-1}}.
$$
Hence, the number of such prime factors is bounded from above by
$$
\frac{\sigma\deg(f)}{(\alpha' \sigma-\dim({\sf G}))\tau s^{-1}}
$$
provided that $q$ is sufficiently large.
Moreover, since $b$ and $q$ are coprime, $f(\gamma)$ is not divisible
by any prime which divides $q$. Hence,
the number of prime factors of $f(\gamma)$
(with multiplicities)
is bounded by 
$$
r>\frac{\sigma\deg(f)}{(\alpha \sigma-\dim({\sf G}))\tau_0 (9t(f))^{-1}}=\frac{9\alpha \sigma}{\alpha \sigma-\dim({\sf G})}\cdot \sigma_0 \cdot t(f)\deg(f).
$$
Hence, every $\gamma$ counted in $S(T,q,X^{\tau/s})$ satisfies conclusion (i) of the theorem as well.

We have shown that every $\gamma$ counted in $S(T,q,X^{\tau/s})$
satisfies (i) and (ii). Since it follows from (\ref{eq:quant 1}) and
(\ref{eq:quant 2}) that $S(T,q,X^{\tau/s})\ge 1$ for sufficiently large $q$,
this completes the proof. 
\end{proof}

In order to complete the proof of the theorem, we need to show the following

\begin{lemma}\label{l:a}
$\sum_{p|q} \frac{\log p}{p}=O(\log\log q)$.
\end{lemma}
\begin{proof}
It is sufficient to prove the claim for square-free $q$.
Moreover, since the function $p\mapsto (\log p)/p-c_1\log(p+c_2)$,
$c_1,c_2>0$, is decreasing for $p\ge 3$, it remains to verify the estimate when $q$ is product of all consecutive primes
less than $z$. In this case, by \cite[Th.~2.7(b)]{mv},
$$
\sum_{p|q} \frac{\log p}{p}=O(\log z),
$$
and by the Prime Number Theorem,
$$
\log q=\sum_{p\le z} \log p\sim z,
$$
which implies the claim. 
\end{proof}

We note the proof of Theorem \ref{th:lin group} not only implies existence 
of solutions for congruences, but also gives the following quantitative 
estimate.

\begin{theorem}\label{th:linnik quant}
Under the notation of Theorem \ref{th:lin group},
\begin{align*}
\left|\left\{x\in {\sf G}(\mathbb{Z}):\,   
\begin{tabular}{l}
\hbox{$f(x)$ has at most $r$ prime factors}\\
\hbox{$f(x)=b\,\hbox{\rm mod}\, q$ and $\|x\|\le q^\sigma$}
\end{tabular}
\right\}\right|
\gg_{\sigma}
  \frac{1}{|f({\sf G}(\mathbb{Z}))\,\hbox{\rm mod}\, q|}\cdot 
\frac{N_{q^{\sigma}}({\sf G}(\mathbb{Z}))}{(\log q)^{t(f)}}
\end{align*}
for sufficiently large $q$.
\end{theorem}

\begin{proof}
Since by (\ref{eq:quant 1}) and Theorem \ref{th:S-arith}, 
for every $\gamma_0\in \Gamma_q$ and sufficiently large $q$,
$$
X=|\gamma_0\Gamma_q\cap B_{q^\sigma}|\gg_\sigma \frac{N_{q^{\sigma}}({\sf G}(\mathbb{Z}))}{|\Gamma:\Gamma_q|},
$$
the claim of the theorem follows from (\ref{eq:quant 2}) by summing
over $\gamma_0\in \Gamma/\Gamma_q$ such that $f(\gamma)=b\,\hbox{mod}\, q$.
\end{proof}

\begin{proof}[Proof of Theorem \ref{th:linnik}]
If ${\sf G}$ is anisotropic over $\mathbb{R}$, then $\Gamma$ is finite.
Hence, we may assume that ${\sf G}$ is isotropic.
We apply Theorem \ref{th:lin group} with the function 
$\tilde f: {\sf G}\to \mathbb{C}$ given by $\tilde f(g)=f(gv)$.
Since $\|\gamma v\|\ll \|\gamma\|$, the claim of Theorem \ref{th:linnik} follows.
\end{proof}

\begin{proof}[Proof of Theorem \ref{th:linnik2}]
We apply the argument of the proof of Theorem \ref{th:lin group}
with the sets $S_T\subset {\sf G}(\mathbb{R})$ introduced in Section \ref{sec:affine}
(in place of the sets $B_T$)
and the polynomial function $\tilde f$
on ${\sf G}$ defined by $\tilde f(x)=f(xv)$.
Using the estimate on $|\delta \Gamma_{q}(dN_1)\cap S_T|$
provided by Theorem \ref{th:GN2}, this argument carries out with no changes.
Let $T=q^\sigma$ with $\sigma$ as in Theorem \ref{th:lift quant sym}.
We conclude from (\ref{eq:vol2}) that for sufficiently large $q$,
\begin{equation}\label{eq:quant 11}
X=|\gamma_0\Gamma_{q}\cap S_T|\gg_\sigma \frac{\vol(S_T)}{|\Gamma:\Gamma_q|}
 \gg q^{\alpha \sigma-\dim({\sf G})},
\end{equation}
where $\alpha$ is as in (\ref{eq:vol2}), and 
\begin{equation}\label{eq:quant 22}
S(T,q,X^{\tau/s})\gg_\sigma \frac{X}{(\log X)^{t(f)}},
\end{equation}
where $\tau<\tau_0=(2n_e(p))^{-1}\dim({\sf G})^{-1}(1+3\dim({\sf G}))^{-1}$.
As in the proof of Theorem \ref{th:lin group}, we conclude that
every $\gamma$ counted in $S(T,q,X^{\tau/s})$ is a product of at most $r$ factors, where
\begin{align}\label{eq:rrr}
r&>\frac{\sigma\deg(f)}{(\alpha \sigma-\dim({\sf G}))\tau_0 (9t(f))^{-1}}\\
&=\frac{\sigma}{\alpha \sigma-\dim({\sf G})} 9t(f)\deg(f) \dim({\sf G})(1+3\dim({\sf G}))2n_e(p).\nonumber
\end{align}

Now using (\ref{eq:quant 11}) and (\ref{eq:quant 22}),
for every $\gamma_0\in\Gamma$ such that $f(\gamma_0 v)=b\,\hbox{mod}\, q$, we have the estimate
\begin{align*}
\left|\left\{\gamma\in \gamma_0\Gamma_q\cap S_T:\,   
\hbox{$f(x)$ has at most $r$ prime factors}
\right\}\right|\\\gg_\sigma \frac{X}{(\log X)^{t(f)}}
\gg_\sigma \frac{1}{|\Gamma:\Gamma_q|}\cdot \frac{|\Gamma\cap
S_T|}{(\log q)^{t(f)}}.
\end{align*}
Since every $\gamma\in \gamma_0\Gamma_q$ satisfies 
$f(\gamma v)=b\,\hbox{mod}\, q$, we conclude that
\begin{align*}
&\left|\left\{\gamma\in \Gamma\cap S_T:\,   
\begin{tabular}{l}
\hbox{$f(\gamma v)$ has at most $r$ prime factors}\\
\hbox{$f(\gamma v)=b\,\hbox{\rm mod}\, q$}
\end{tabular}
\right\}\right|\\
\gg_\sigma &\;\;
\frac{|\{\gamma\in \Gamma/\Gamma_q :\, f(\gamma v)=b\,\hbox{\rm mod}\, q
\}|}{|\Gamma:\Gamma_q|}\cdot \frac{|\Gamma\cap
S_T|}{(\log q)^{t(f)}}\\
= &\;\;\frac{1}
{|f(\mathcal{O})\,\hbox{\rm mod}\, q|}
\cdot \frac{|\Gamma\cap
S_T|}{(\log q)^{t(f)}}.
\end{align*}
Since the cardinality of the fibers of the map
$\pi:\Gamma\cap S_T\to \cO(T):\gamma\mapsto \gamma v$
is uniformly bounded, we conclude that
\begin{align*}
\left|\left\{x\in\mathcal{O}(T):\,   
\begin{tabular}{l}
\hbox{$f(x)$ has at most $r$ prime factors}\\
\hbox{$f(x)=b\,\hbox{\rm mod}\, a$}
\end{tabular}
\right\}\right|
\gg_\sigma \frac{1}
{|f(\mathcal{O})\,\hbox{\rm mod}\, q|}
\cdot \frac{|\Gamma\cap
S_T|}{(\log q)^{t(f)}},
\end{align*}
which implies the theorem because of (\ref{eq:vol1}) and  (\ref{eq:vol3}).
\end{proof}

\end{document}